\documentclass{amsart}

\usepackage[margin=1cm]{geometry}
\usepackage{lipsum} 
\usepackage{amsmath, amssymb, amsthm}
\usepackage{mathrsfs, color}
\usepackage{url}
\usepackage{stackrel} 
\usepackage{courier}
\usepackage[all,cmtip]{xy} 
\usepackage{multicol,arydshln}
\usepackage{fullpage}
\usepackage{cancel,comment} 
\usepackage[usenames,dvipsnames]{xcolor}
\usepackage{parcolumns}

\usepackage[colorlinks=true, pdfstartview=FitV, linkcolor=blue, citecolor=blue, urlcolor=blue]{hyperref}	

\definecolor{darkred}{rgb}{0.7,0,0} 
\newcommand{\defn}[1]{{\color{darkred}\emph{#1}}} 

\definecolor{darkgreen}{rgb}{0,0.6,0} 
\definecolor{darkblue}{rgb}{0,0,0.6} 

\catcode`~=11 
\newcommand{\urltilde}{\kern -.15em\lower .7ex\hbox{~}\kern .04em}  
\catcode`~=13 

\newcommand{\ZZ}{\mathbb{Z}}

\newcommand{\CC}{\mathbb{C}}
\newcommand{\NN}{\mathbb{N}}

\newcommand{\bfa}{\mathbf{a}}
\newcommand{\bfb}{\mathbf{b}}

\newcommand{\bfr}{\mathbf{r}}

\newcommand{\fD}{\mathfrak{D}}
\newcommand{\fM}{\mathfrak{M}}
\newcommand{\fP}{\mathfrak{P}}
\newcommand{\fb}{\mathfrak{b}}
\newcommand{\fg}{\mathfrak{g}}
\newcommand{\fu}{\mathfrak{u}}
\newcommand{\fp}{\mathfrak{p}}
\newcommand{\mcI}{\mathcal{I}}
\newcommand{\mcO}{\mathcal{O}}

\newcommand{\iso}{\simeq}
\newcommand{\orbit}[1]{\mathcal{O}_{#1}}
\DeclareMathOperator{\GL}{GL}

\DeclareMathOperator{\supp}{supp}
\DeclareMathOperator{\rk}{rk}  
\DeclareMathOperator{\Orb}{Orb}  

\renewcommand{\mid}{:}

\newcommand{\diag}{\mathop{\mathrm{diag}}\nolimits}

\newcommand{\length}{\mathop{\mathrm{length}}\nolimits}

\newcommand{\Ad}{\mathop{\mathrm{Ad}}\nolimits}
\newcommand{\ad}{\mathop{\mathrm{ad}}\nolimits}

\newcommand{\red}{\mathop{\mathrm{red}}\nolimits} 
\newcommand{\Spec}{\mathop{\mathrm{Spec}}\nolimits} 

\newcommand{\Hilb}{\mathop{\mathrm{Hilb}}\nolimits}

\newcommand{\I}{\mathop{\mathbf{I}}\nolimits}

\newcommand{\spec}{\mathop{\mathrm{Spec}}\nolimits}  

\newcommand{\pr}{\mathop{\mathrm{pr}}\nolimits} 
\newcommand{\tr}{\mathop{\mathrm{tr}}\nolimits}

\newcommand{\Lie}{\mathop{\mathrm{Lie}}\nolimits}

\newcommand{\End}{\mathop{\mathrm{End}}\nolimits}
\newcommand{\FHilb}{\mathop{\mathrm{FHilb}}\nolimits} 
\newcommand{\PFHilb}{\mathop{\mathrm{PFHilb}}\nolimits}

\makeatletter
\newif\if@borderstar
   \def\bordermatrix{\@ifnextchar*{%
       \@borderstartrue\@bordermatrix@i}{\@borderstarfalse\@bordermatrix@i*}%
   }
   \def\@bordermatrix@i*{\@ifnextchar[{\@bordermatrix@ii}{\@bordermatrix@ii[()]}}
   \def\@bordermatrix@ii[#1]#2{%
   \begingroup
     \m@th\@tempdima8.75\p@\setbox\z@\vbox{%
       \def\cr{\crcr\noalign{\kern 2\p@\global\let\cr\endline }}%
       \ialign {$##$\hfil\kern 2\p@\kern\@tempdima & \thinspace %
       \hfil $##$\hfil && \quad\hfil $##$\hfil\crcr\omit\strut %
       \hfil\crcr\noalign{\kern -\baselineskip}#2\crcr\omit %
       \strut\cr}}%
     \setbox\tw@\vbox{\unvcopy\z@\global\setbox\@ne\lastbox}%
     \setbox\tw@\hbox{\unhbox\@ne\unskip\global\setbox\@ne\lastbox}%
     \setbox\tw@\hbox{%
       $\kern\wd\@ne\kern -\@tempdima\left\@firstoftwo#1%
         \if@borderstar\kern2pt\else\kern -\wd\@ne\fi%
       \global\setbox\@ne\vbox{\box\@ne\if@borderstar\else\kern 2\p@\fi}%
       \vcenter{\if@borderstar\else\kern -\ht\@ne\fi%
         \unvbox\z@\kern-\if@borderstar2\fi\baselineskip}%
         \if@borderstar\kern-2\@tempdima\kern2\p@\else\,\fi\right\@secondoftwo#1 $%
     }\null \;\vbox{\kern\ht\@ne\box\tw@}%
   \endgroup
   }
\makeatother

\theoremstyle{plain}
\newtheorem{theorem}{Theorem}[section] 
\newtheorem{proposition}[theorem]{Proposition}
\newtheorem{lemma}[theorem]{Lemma}
\newtheorem{conjecture}[theorem]{Conjecture} 
\newtheorem{corollary}[theorem]{Corollary}

\theoremstyle{remark}

\newtheorem{remark}[theorem]{Remark} 

\newtheorem{example}[theorem]{Example}

\usepackage[colorinlistoftodos]{todonotes}

\begin{document} 

\title[Almost-commuting partial GS resolutions and CM systems]{The regularity of almost-commuting partial Grothendieck--Springer resolutions and parabolic analogs of Calogero--Moser varieties} 
 
\author{Mee Seong Im}
\address{Department of Mathematical Sciences, United States Military Academy, West Point, NY 10996, USA}
\curraddr{Department of Mathematics, United States Naval Academy, Annapolis, MD 21402, USA}
\email{meeseongim@gmail.com}
\email{im@usna.edu}

\author{Travis Scrimshaw}
\address{School of Mathematics and Physics, The University of Queensland, St.\ Lucia, QLD 4072, Australia}
\email{tcscrims@gmail.com}


\date{\today}

\begin{abstract}
Consider the moment map $\mu \colon T^*(\mathfrak{p} \times \mathbb{C}^n) \to \mathfrak{p}^*$ for a parabolic subalgebra $\mathfrak{p}$ of $\mathfrak{gl}_n(\mathbb{C})$.
We prove that the preimage of $0$ under $\mu$ is a complete intersection when $\mathfrak{p}$ has finitely many $P$-orbits, where $P\subseteq \operatorname{GL}_n(\mathbb{C})$ is a parabolic subgroup such that $\operatorname{Lie}(P) = \mathfrak{p}$, and give an explicit description of the irreducible components.
This allows us to study nearby fibers of $\mu$ as they are equidimensional, and  one may also construct GIT quotients $\mu^{-1}(0) /\!\!/_{\chi} P$ by varying the stability condition $\chi$.  Finally, we study a variety analogous to the scheme studied by Wilson with connections to a Calogero--Moser phase space where only some of particles interact.
\end{abstract}

\maketitle

\section{Introduction}\label{section:intro} 

The classical Springer resolution and the Grothendieck--Springer resolution are foundational and important schemes in algebraic geometry and representation theory. 
The Springer resolution is a desingularization of the variety of nilpotent elements in a semisimple Lie algebra (see, \textit{e.g.},~\cite{MR2838836,MR739636}), while the Grothendieck--Springer resolution is the minimal, symplectic resolution of the variety of the semisimple Lie algebra, which contains the Springer resolution (see, \textit{e.g.},~\cite{MR2838836,MR3836769}).

We consider the case for $G = \GL_n(\CC)$, the general linear group of invertible complex $n \times n$ matrices, and let $\fg = \mathfrak{gl}_n(\CC) = \Lie(G)$, the Lie algebra of $G$ of all complex $n \times n$ matrices. 
Under the adjoint action of $G$ on $\fg$ (\textit{i.e.}, $G$ acts by conjugation), we have the natural adjoint quotient map $\rho \colon \fg \twoheadrightarrow \fg /\!\!/ G \cong \CC^n$, which sends $r$ to the tuple of coefficients of its characteristic polynomial $\chi_r(t)$. So we have that $r \mapsto \bigl( \tr(r), \dotsc, \det(r) \bigr)$, and these polynomials are invariant under the adjoint action.
Since $\rho^{-1}(0)$ consists of those $r\in \fg$ such that $\chi_r(t) = t^n$, the preimage $\rho^{-1}(0)$ is precisely the nilpotent elements in $\fg$.  We denote $\mathcal{N} := \rho^{-1}(0)$, the nilpotent cone of $\fg$. 

Let $P$ be a standard parabolic subgroup of $G$, which consists of invertible block upper triangular matrices, and let $\fp = \Lie(P)$, the Lie algebra of $P$, consisting of all block upper triangular matrices.
In particular, when all blocks have size $1$ then $P$ is the standard Borel subgroup $B$ of invertible upper triangular matrices, and 
$\fp$ is the standard Borel subalgebra $\fb$ of all upper triangular matrices. 
Let $\mathcal{N}_{P} =\fp \cap \mathcal{N}$ be the nilpotent cone of $\fp$. 
Define $G/P$ to be the partial flag (projective) variety parameterizing parabolic subalgebras in $\fg$. 
Let
\begin{align*}
\widetilde{\fg}_{P}&  := G \times_P \fp  = \{ (x,\fp' )\in \mathfrak{g}\times G/P \mid x \in \fp' \},
\\ \widetilde{\mathcal{N}}_{P} & := G \times_P \mathfrak{n}= \{ (x, \fp') \in \mathcal{N}_P \times G/P \mid x\in \fp' \}.
\end{align*}
Let $p \colon \widetilde{\mathcal{N}}_{P} \twoheadrightarrow \mathcal{N}_{P}$ 
be the partial Springer resolution   
(see \cite[Chapter $3$]{MR2838836} and \cite{MR739636} for the case of a full flag)  
and $\pi \colon \widetilde{\fg}_{P} \twoheadrightarrow \fg$ 
be the partial Grothendieck--Springer resolution (\cite{MR0442103,MR0352279,MR0430094}).
Then we have an inclusion 
$\iota \colon \widetilde{\mathcal{N}}_{P} \hookrightarrow \widetilde{\fg}_{P}$ 
such that 
$\pi \circ \iota = \gamma \circ p$, where 
$\gamma \colon \mathcal{N}_{P} \hookrightarrow \fg$  
is the natural inclusion. 
The nilpotent cone is resolved by the cotangent bundle $T^* (G/P)$, which is a closed subvariety of $\widetilde{\mathcal{N}}_P$.

In this manuscript, we consider a moment map associated to the partial Grothendieck--Springer resolution. We show that the preimage of $0$ under this map is a complete intersection when the parabolic subgroup has at most $5$ diagonal blocks. Such property is important and interesting since it provides the exact number of irreducible components in the special fiber. Furthermore, one can also construct topological fibers, analogous to the topological Springer fibers of type $A$ corresponding to nilpotent endomorphisms with two equally-sized Jordan blocks studied in \cite{MR1928174,MR2078414}, and study the topology and intersection of the cotangent bundle of the partial Grothendieck--Springer fibers. 
Let us make precise the construction of this moment map for the parabolic setting.

Using the trace pairing, we have the identifications $\fg \iso \fg^*$ and $\fp^* \iso \fg / \fu$, where $\fu$ is the maximal nilpotent subalgebra of $\fp$, \textit{i.e.}, strictly block upper triangular matrices. The latter pairing is given by $\fp\times \fg \to \CC$, $(r,s) \mapsto \tr(rs)$, where this map factors through the bilinear, nondegenerate pairing $\fp \times \fg / \fu \to \CC$.

Let $P$ act on the space $\fp \times \CC^n$, which we differentiate to obtain the comoment map
\[
\mu \colon T^*(\fp \times \CC^n) \to \fp^*.
\]
We can write the map $\mu$ as $(r, s, i, j) \mapsto \overline{[r,s] + ij}$ by identifying $T^*(\fp \times \CC^n) \cong \fp \times \fg/\fu \times \CC^n \times (\CC^n)^*$.
Note that since $P$ acts on $\fp \times \CC^n$, the $P$-action is induced onto $T^*(\fp \times \CC^n)$.
Now, let $G$ act on the first and the third factors of $G\times \fp\times \CC^n$, which also induces a moment map
\[
\mu_G \colon T^*(G \times \fp \times \CC^n) \to \fg^*.
\]
We also have a natural $P$-action on the second and the third factors of $G \times \fp \times \CC^n$, which is induced onto the cotangent bundle with moment map $\mu_P \colon T^*(G \times \fp \times \CC^n) \to \fp^*$.
Thus there is a natural $G\times P$-action on $G \times \fp \times \CC^n$. 
We combine the two maps $\mu_G$ and $\mu_P$ to obtain the moment map 
\[
\mu_{G\times P} \colon T^*(G\times \fp\times \CC^n) \to (\fg \times \fp)^*.
\]

Next, note that $T^*(G\times \fp\times \CC^n) \cong G\times \fg^* \times \fp\times \fp^*\times \CC^n \times (\CC^n)^*$, and consider
\begin{equation}
\label{eq:GxP_moment_zero_set}
\mu_{G\times P}^{-1}(0,\overline{0}) = \{ (g,\theta,r,s,i,j)\in T^*(G\times \fp \times \CC^n) \mid \theta = ij, \overline{g \theta g^{-1}} = -\ad_r^*(s) \}. 
\end{equation}
We may set $g=1$ since $g\in G$ is a free variable. 
Then there is a bijection between $P$-orbits on $\mu^{-1}(\overline{0})$ and  $G\times P$-orbits on $\mu_{G\times P}^{-1}(0,\overline{0})$, giving us an isomorphism between the quotient stacks $\mu^{-1}(\overline{0})/P$ and $T^*(\widetilde{\fg}_{P}\times \CC^n/G)$ (see~\cite[Cor.~3.3]{Nevins-GSresolutions}, and~\cite[Prop.~1.1]{MR3836769} when $P=B$). 
This gives us a connection between the Hamiltonian reduction of the enhanced $P$-moment map and the partial Grothendieck--Springer resolution. 
That is, studying the $G$-orbits on $T^*(\widetilde{\fg}_{P}\times \CC^n)$ is equivalent to studying the $P$-orbits on $T^*(\fp \times \CC^n)$.

\begin{theorem}\label{theorem:Grothendieck--Springer-resolution-complete-intersection} 
Let $\alpha = (\alpha_1, \dotsc, \alpha_{\ell})$ such that $\ell \leq 5$.
Let $P$ be the parabolic subgroup of $\GL_n(\CC)$ with block size vector $\alpha$.
The components of $\mu^{-1}(\overline{0})$ form a complete intersection.  
\end{theorem}

Next, we give a description of the irreducible components. 
Consider the scheme
\begin{equation}
\fM := \{(r,s,i,j) \in \fp \times \fp^* \times \CC^{n} \times (\CC^n)^* \mid \overline{[r,s] + ij} = \overline{0}\} = \mu^{-1}(\overline{0}), 
\end{equation}
where $ij$ denotes the product with $i$ as a column vector and $j$ as a row vector.
Let $\fD$ be the representatives $(r,s,i,j) \in \fM$ of the $P$-orbits, denoted by $\orbit{(r,s,i,j)}$, that are given by Lemma~\ref{lemma:rep_distinct_eigenvalues}.
We also require the set $\bfa_{\downarrow}$, which consists of the first $a_i$ entries of the $i$-th block of $P$ (see~\eqref{eq:updown_sets}).
Let $\supp(i)$ denote the support of $i \in \CC^n$.

\begin{theorem}\label{thm:irred-comp-itemized}
Let $\alpha = (\alpha_1, \dotsc, \alpha_{\ell})$ such that $\ell \leq 5$.
Let $P$ be the parabolic subgroup of $\GL_n(\CC)$ with block size vector $\alpha$.
Choose some $\bfa = (a_1, \dotsc, a_{\ell}) \in \NN^{\ell}$ such that $a_k \leq \alpha_k$ for all $1 \leq k \leq \ell$.
The $\prod_{i=1}^{\ell} (\alpha_i+ 1)$ irreducible components of $\fM$ are the closures of
\[
\fM'_{\bfa} = \bigcup \left\{ \orbit{(r,s,i,j)} \mid (r,s,i,j) \in \fD \text{ with } \begin{gathered} \supp(i) = \bfa_{\downarrow}, \\ \supp(j) = \{1, \dotsc, n\} \setminus \bfa_{\downarrow} \end{gathered} \right\}.
\]
Moreover, $\fM$ is reduced and equidimensional with $\dim \fM = d_{\fp} + 2n$, where $d_{\fp} = \dim \fp$.
\end{theorem}

Let us remark about why we must have at most $5$ blocks in $P$.
Our proof of Theorem~\ref{theorem:Grothendieck--Springer-resolution-complete-intersection} follows~\cite{GG06}, but the following theorem of  L.~Hille and G.~R\"ohrle is crucial in replacing~\cite[Lemma~2.1]{GG06} in the parabolic setting.

\begin{theorem}[{\cite[Thm.~4.1]{HR99}}]
\label{thm:Hille-Roehrle-block-vec}
Let $k$ be an infinite field. 
Let $P$ be a parabolic subgroup of $\GL_n(k)$ with block size vector $\alpha = (\alpha_1, \dotsc, \alpha_{\ell})$. 
Let $U$ be the unipotent radical of $P$, and $\fu = \Lie(U)$, the Lie algebra of $U$. 
Then the number of $P$-orbits on $U$ or $\fu$ is finite if and only if $\ell \leq 5$.
\end{theorem}

We also require an analogous property of how for some $b \in \mathfrak{gl}_n(\CC)$, the Jordan canonical form of $b$ is a similar matrix (\textit{i.e.}, in the $G$-orbit of $b$) that naturally acts on the decomposition of $\CC^n$ into the generalized eigenspaces under the action of $b$.
Indeed, we give an algorithm that constructs one such matrix in the $P$-orbit, which we coin a Jordan $P$-semicanonical form of $b$.

For the case of $P = B$, T.~Nevins finds in~\cite{Nevins-GSresolutions} at least $2^n$ irreducible components in the locus $\mu^{-1}(\overline{0})$, which is defined by $n(n+1)/2$ equations, and conjectures that $\mu^{-1}(\overline{0})$ is a complete intersection. 
The first author has proved this conjecture when $\mu$ is restricted to its semisimple locus in~\cite{MR3312842,MR3836769}.
In this manuscript, we prove Nevins's conjecture for the entire preimage $\mu^{-1}(\overline{0})$ up to rank $5$ matrices due to the restriction in Theorem~\ref{thm:Hille-Roehrle-block-vec}.
Moreover, we note that Theorem~\ref{thm:irred-comp-itemized} previously appeared in~\cite{Nevins-GSresolutions} under the assumption that $\mu^{-1}(\overline{0})$ was a complete intersection.
The proof that we give of Theorem~\ref{thm:irred-comp-itemized} is different from the one given in~\cite{Nevins-GSresolutions} as we give an analogous proof to~\cite{GG06} instead of lifting up to the $\fg$ setting from $\fb$.

We note that Theorem~\ref{theorem:Grothendieck--Springer-resolution-complete-intersection} allows one to construct affine and geometric invariant theory (GIT) quotients (see, \textit{e.g.},~\cite{MR1315461, MR2537067}), and it would be interesting to show their connections to the Hilbert scheme $\Hilb^{n}(\CC^2)$ of $n$ points on a complex plane (see, \textit{e.g.},~\cite{GG06,MR1711344}), which is described as 
\[ 
\Hilb^n(\CC^2) = 
\{ 
I \subseteq \CC[r,s] \mid I \mbox{ is an ideal satisfying } \length(V(I)) = \dim_{\CC} (\CC[r,s]/I)=n 
\}, 
\]   
and the isospectral Hilbert scheme~\cite[Sec.~3.2]{MR1839919}, which is the reduced fiber product 
\[
(\Hilb^n(\CC^2)\times_{S^n\CC^2} (\CC^2)^n)^{\red},
\]  
where $S^n (\CC^2) = (\CC^2)^n / S_n$, unordered $n$-tuples of points in $\CC^2$. 
For notational purposes, we identify the closed points of $\Hilb^n(\CC^2)$ with ideals $I\subseteq \CC[r,s]$ satisfying $\dim_{\CC} (\CC[r,s]/I) = n$. 
That is, the Hilbert scheme $\Hilb^n(\CC^2)$ parameterizes finite closed subschemes of length $n$, while the reduced closed points of the isospectral Hilbert scheme are the tuples $(I,p_1,\ldots, p_n) \in \Hilb^n(\CC^2) \times (\CC^2)^n$ such that $\pi(I) = [[ p_1,\ldots, p_n]]$, where $\pi \colon \Hilb^n(\CC^2) \twoheadrightarrow S^n\CC^2$ is the Hilbert--Chow morphism associating a closed subscheme with its corresponding cycle. 
It would be interesting to construct a morphism between 
the Hamiltonian reduction of our parabolic moment map and noncommutative products of the Hilbert scheme $\Hilb^{\alpha_p}(\CC^2)$, 
as well as investigate connections between the isospectral Hilbert scheme and our variety.

It would also be interesting to show (rational) morphisms between the Hamiltonian reduction of our parabolic moment map and the partial flag Hilbert scheme on a complex plane 
\[ 
\PFHilb^{\ell}(\CC^2) = 
\left\{ 
I_{\ell} \subseteq \ldots \subseteq I_1 \subseteq I_0 = \CC[r,s] \mid \dim_{\CC} \frac{\CC[r,s]}{I_k} = \sum_{j=1}^k \alpha_j \right\}, 
\]  
which also may be considered in terms of lower block triangular matrices $P^-$ in $G$, \textit{i.e.}, given $\fp^- = \Lie(P^-)$ with $\mathfrak{u}^-$ as the maximal nilpotent subalgebra of $\fp^-$, 
\[ 
\PFHilb^n(\CC^2) = \{ (r,s,i)\in \fp^-\times \fp^-\times \CC^n \mid [r,s]=0, r^a s^bi \mbox{ span } \CC^n\}/P^-, 
\] 
 especially since both constructions involve $P$ (or $P^-$)-conjugation action on its subalgebra. 
Partial flag Hilbert schemes are fascinating in their own right, and are of great interest in categorical representation theory and quantum topology since (complete) flag Hilbert schemes give correspondence between Koszul complexes of the torus fixed points on the flag Hilbert scheme $\FHilb^n(\CC^2)$ and idempotents in the category of Soergel bimodules (see, \textit{e.g.},~\cite{gorsky2016flag}).

Our last main result is about the varieties analogous to the ones in~\cite{Wilson98}.
Indeed, let
\[
\overline{C}_n = \{ (r,s) \in \fp \times \fp^* \mid \rk([r,s] + \I_n) = 1 \},
\]
where $\I_n$ is the $n \times n$ identity matrix.
Let $C_n = \overline{C}_n / P$, where $P$ acts on the pair $(r,s)$ by a simultaneous conjugation. 
Let $C_n' = \{ [(r,s)] \in C_n \mid r \mbox{ is diagonalizable}\}$, which is a subspace of $C_n$. 
We show that $r \in C_n'$ implies that the eigenvalues of $r$ must be pairwise distinct (see Lemma~\ref{lemma:r-diagonalizable}).
Now let $\fP := \fp \times \CC^n$.
Consider the following subvariety 
\begin{equation}\label{eqn:Wilson} 
\widetilde{C}_n = \{ 
(r,s,i,j) \in T^*\fP \mid
\overline{[r,s] + \I_n} = -\overline{ij} \}.
\end{equation} 
We see that $\widetilde{C}_n = \mu^{-1}(-\I_n)$, where $\mu$ is the moment map~\eqref{eq:moment_map}.

\begin{remark}
Our moment map and subvariety $\widetilde{C}_n$ differs from the moment map and subvariety (after replacing $P$ with $G$) considered in~\cite{Wilson98} by $i \mapsto -i$. We do this to match the moment map from~\cite{GG06,Nevins-GSresolutions} when $P = B$ (see~\eqref{eq:moment_map} above). Note that this sign difference comes from the fact that~\cite[page $9$, lines $8$-$9$]{Wilson98} uses $\fg \times (\CC^n)^*$ 
in his study of Calogero--Moser systems 
instead of $\fg \times \CC^n$ as in~\cite{GG06}. 
\end{remark}

Note that for fixed $(r,s) \in \fp \times \fp^*$, there exist vectors $i$ and $j$ satisfying~\eqref{eqn:Wilson} if and only if $(r,s) \in \overline{C}_n$. 
That is, 
we claim that every rank $1$ matrix is of the form $ij$ for some column and row vector $i$ and $j$, respectively. 
Since $\tr([r,s] + \I_n) = n$, we must have $1 \leq \rk([r,s] + \I_n) = \rk(ij) \leq 1$ from~\eqref{eqn:Wilson}. Thus we have $\rk([r,s] + \I_n) = 1$, and we can compute $i$ and $j$ up to a common scalar multiple $\lambda$ by $i \mapsto \lambda i$ and $j \mapsto \lambda^{-1}j$. Therefore, we can identify $\overline{C}_n$ with the quotient of $\widetilde{C}_n$ by the scalar matrices in $P$ under the action~\eqref{eq:B_action}, and we obtain the same space $C_n$ as a quotient of $\widetilde{C}_n$ or of $\overline{C}_n$, \textit{i.e.}, since the $P$-action given by~\eqref{eq:B_action} above preserves $\widetilde{C}_n$, we will also write $C_n = \widetilde{C}_n / P$.

\begin{theorem}
\label{theorem:CM-smooth-irreducible-affine-var}
Let $\alpha = (\alpha_1, \dotsc, \alpha_{\ell})$ such that $\ell \leq 5$.
The scheme $C_n$ is a smooth irreducible affine algebraic variety of dimension $2n$. 
\end{theorem}

Our proof of Theorem~\ref{theorem:CM-smooth-irreducible-affine-var} is given by extending the techniques of~\cite{Wilson98}, and again Theorem~\ref{thm:Hille-Roehrle-block-vec} and the Jordan $P$-semicanonical form play key roles.

Let us briefly discuss the relationship with Calogero--Moser particle systems (see, \textit{e.g.},~\cite{Calogero71,KKS78,Moser75,Moser81}).
We show that there exists a representative for every point $(r, s, i, j) \in C'_n$ such that $r = \diag(-\rho_1, \dotsc, -\rho_n)$ and
\[
\tr s^2 = \sum_{i=1}^n \sigma_i^2 - 2 \sum_{p<q \in I_k} (\rho_i - \rho_j)^{-2},
\]
where $I_k = \{\alpha_1 + \ldots + \alpha_{k-1}+1, \dotsc, \alpha_1 + \ldots + \alpha_k\}$, and in particular $\rho_i \neq \rho_j$ for all $i,j$ (Lemma~\ref{lemma:r-diagonalizable}).
Thus, $\frac{1}{2} \tr s^2$ is the sum (over $1 \leq k \leq \ell$) of Hamiltonians of Calogero--Moser particle systems with $\alpha_k$ particles. Equivalently, this is the Hamiltonian of a Calogero--Moser particle system where only certain particles (depending on $\alpha$) interact with each other. 
We remark that the existence of such a representative does not require that there are at most $5$ blocks in $P$.

We also conjecture that our results  do not require $\ell \leq 5$ (\textit{i.e.}, not relying on Theorem~\ref{thm:Hille-Roehrle-block-vec}); thus extending Nevins's conjecture for all parabolic subgroups.

\begin{conjecture}
Theorem~\ref{theorem:Grothendieck--Springer-resolution-complete-intersection}, Theorem~\ref{thm:irred-comp-itemized}, and Theorem~\ref{theorem:CM-smooth-irreducible-affine-var} hold for all compositions $\alpha$.
\end{conjecture}

This paper is organized as follows.
In Section~\ref{section:prelim}, we give some background on $P$-invariant functions on $\fp$. 
In Section~\ref{section:Jordan-P-form}, we give a notion of Jordan $P$-semicanonical form for matrices, which are crucial in the proof for Theorem~\ref{theorem:Grothendieck--Springer-resolution-complete-intersection}.
We prove Theorem~\ref{theorem:Grothendieck--Springer-resolution-complete-intersection} in Section~\ref{section:proof-gs-resoln-ci}, and 
we prove Theorem~\ref{thm:irred-comp-itemized} in Section~\ref{section:proof-irred-comp-itemized}.
We prove Theorem~\ref{theorem:CM-smooth-irreducible-affine-var} in Section~\ref{section:connections-to-CM-particles}.

\section*{Acknowledgements}

The first author would like to thank Thomas Nevins, Milen Yakimov, and Ben Elias for fruitful discussions.
The authors thank SageDays79 for the invitation and the Hebrew University of Jerusalem for the excellent working environment where this paper was initiated. 
T.S.\ would like to thank the United States Military Academy for its hospitality during his visit in October, 2018, where his visit was partially supported by M.S.I.'s AMS-Simons Travel Grant which is administered by the American Mathematical Society, with support from the Simons Foundation. 
The authors thank one of the referee for bringing a certain theorem, attributed to Kashin~\cite{Kashin90}, in the manuscript~\cite{PR97} by V.~Popov and G.~R\"ohrle to our attention on an earlier version of this manuscript. This led us to discover~\cite{HR99} by L.~Hille and G.~R\"ohrle and to correctly prove our results in all possible parabolic settings.
The authors also thank the other referee for the useful comments.
M.S.I.\ was supported by United States Military Academy and Hebrew University of Jerusalem.
T.S.\ was partially supported by the Australian Research Council grant DP170102648. 
This work benefited from computations using \textsc{SageMath}~\cite{sage}.

\section{Preliminaries}
\label{section:prelim}

Throughout this paper, we consider all groups and (Lie) algebras to be over $\CC$ unless otherwise stated.
Let $\NN$ be the set of all nonnegative integers (in particular, $0 \in \NN$).
Let $S_n$ denote the symmetric group on $n$ letters, and we will identify $S_n$ with the subgroup of all permutation matrices of $\GL_n$.
Therefore, we have a natural $S_n$-action on $\CC^n$ given by permuting coordinates.

Let $\alpha = (\alpha_1, \alpha_2, \dotsc, \alpha_{\ell})$ be a composition of $n$, that is $\alpha_i \in \ZZ_{>0}$ for all $i$ and $\sum_{i=1}^{\ell} \alpha_i = n$.
Let $\fp_{\alpha}$ be the set of block upper triangular matrices, where the $i$-th diagonal block has size $\alpha_i$, which is a standard parabolic subalgebra of $\fg = \mathfrak{gl}_n$.
Let $\fu_{\alpha}$ be the set of block strictly upper triangular matrices with block size vector $\alpha$, which is the nilradical of $\fp_{\alpha}$.
Let $\fP_{\alpha} := \fp_{\alpha} \times \CC^n$.
We will identify the cotangent bundle of $\fP_{\alpha}$ as $T^*\fP_{\alpha} \iso \fp_{\alpha} \times \fp_{\alpha}^* \times \CC^n \times (\CC^n)^*$, where $\fp_{\alpha}^* = \fg / \fu_{\alpha}$, and $(\CC^n)^*$ is dual to $\CC^n$.
We will consider vectors in $\CC^n$ as column vectors and vectors in $(\CC^n)^*$ as row vectors.
For an element $x \in \fg$, we write the corresponding element $\overline{x} \in \fg / \fu_{\alpha}$.
Let $P_{\alpha}$ 
be the parabolic subgroup of $G$ corresponding to $\fp_{\alpha}$ (\textit{i.e.}, we have $\Lie(P_{\alpha}) = \fp_{\alpha}$).
To ease notation, we will omit the subscript $\alpha$ on the above notation (\textit{e.g.}, we will simply write $\fp := \fp_{\alpha}$) when there is no danger of confusion.
Note that the standard Borel subalgebra $\fb$ (resp.\ subgroup $B$) of all (resp.\ invertible) upper triangular matrices corresponds to the case when $\alpha = (1, 1, \dotsc, 1)$.

Let $P$ act on the space $\fP$ via $b.(r,i) = (brb^{-1},bi)$.
The comoment map from differentiating the $P$-action is
\[ 
\fp = \Lie(P)\stackrel{a}{\rightarrow}\Gamma(T_{\fP}) \subseteq \CC[T^*\fP], 
\quad 
\mbox{ where }
\quad 
a(v)(r,i) = \frac{d}{dt}\bigl(\exp (tv).(r,i)\bigr) \bigg|_{t=0} = ([v,r], vi). 
\] 
Next, we dualize $a$ to obtain the moment map 
\begin{equation}
\label{eq:moment_map}
\mu \colon T^*\fP \to \fp^*,  
\quad \mbox{ where } 
\quad (r, \overline{s}, i, j) \mapsto \ad_r^*(s) + \overline{a^*(ij)} 
\end{equation}
and $a \colon \fg \to \End(\CC^n)$ is the natural $\fg$-representation on $\CC^n$. 
Since $\fg = \Lie(G)$, we pullback $a$ to obtain $a^* \colon \End(\CC^n)^* \to \fg^*$, where $a^*(ij)=ij$. 
Therefore, there is an induced Hamiltonian $P$-action on $T^*\fP$ given by
\begin{equation}
\label{eq:B_action} 
b.(r,s,i,j) = (\Ad_b(r), \Ad_b^*(s), bi, jb^{-1}) = (brb^{-1}, \overline{bsb^{-1}}, bi, jb^{-1})
\end{equation}
with moment map
\begin{equation}
\begin{split}
\mu \colon T^* \fP & \to \fp^*, 
\\ (r,s,i,j) & \mapsto \overline{[r,s] + ij}.
\end{split}
\end{equation}
There is also a $\fp$-action on $T^*\fP$ given by
\[
q \bullet (r,s,i,j) = ([q,r], [s,q], qi, -jq).
\]

Now, let $G$ act on the first and the third factors of $G\times \fp\times \CC^n$ as $g.(g',r,i) = (g' g^{-1},r,gi)$, which induces the moment map 
\[
\mu_G \colon T^*(G\times \fp \times \CC^n) \cong G \times \fg^* \times \fp \times \fp^* \times \CC^n \times (\CC^n)^* \to \fg^*, 
\quad 
(g,\theta,r,s,i,j) \mapsto -\theta + a^*(ij). 
\] 
We also have a natural $P$-action on the second and the third factors of $G \times \fp \times \CC^n$, given by $b.(g',r,i) = (g',brb^{-1},bi)$. This action is induced onto the cotangent bundle, where its moment map $\mu_P \colon T^*(G \times \fp \times \CC^n) \to \fp^*$ is given by $(g,\theta,r,s,i,j) \mapsto \overline{\Ad_g^*(\theta)} + \ad_r^*(s)$. 
Thus there is a natural $G\times P$-action on $G \times \fp \times \CC^n$. 
We combine the two maps $\mu_G$ and $\mu_P$ to obtain the moment map 
\[
\mu_{G\times P} \colon T^*(G\times \fp\times \CC^n)
\cong G\times \fg^* \times \fp\times \fp^*\times \CC^n \times (\CC^n)^*
\rightarrow (\fg \times \fp)^* \cong \fg^*\times \fp^*,
\]
which is given by 
\[ 
(g,\theta,r,s,i,j) \mapsto (-\theta+ a^*(ij), \: \overline{\Ad_g^*(\theta)}+ \ad_r^*(s) ).   
\] 
Thus, we obtain $\mu_{G\times P}^{-1}(0,\overline{0})$ in~\eqref{eq:GxP_moment_zero_set}.

Consider the scheme  
\[
\fM := \{(r,s,i,j) \in \fp \times \fp^* \times \CC^{n} \times (\CC^n)^* \mid \overline{[r,s] + ij} = 0\}.
\]
Then $\fM$ is equal to $\mu^{-1}(0) \subseteq T^*\fP$. 
Note that the $P$-action preserves $\fM$:
\begin{gather*}
[brb^{-1}, b^{-1}sb] + bijb^{-1} = b \bigl( [rs] + ij \bigr) b^{-1} = 0.
\end{gather*}

Let $\CC^{(k)} := \CC^k / S_k$, so $\CC^{(k)}$ is the set of all unordered tuples (\textit{i.e.}, multisets) of size $k$ of complex numbers.
For the composition $\alpha = (\alpha_1, \dotsc, \alpha_{\ell})$, denote $\CC^{(\alpha)} := \prod_{k=1}^{\ell} \CC^{(\alpha_k)}$.

\begin{lemma}
\label{lem:coadjoint-parabolic}
For the coadjoint action of $P_{\alpha}$ on $\fp_{\alpha}$, we have $\fp_{\alpha} /\!\!/ P_{\alpha} \cong \CC^{(\alpha)} \cong \CC^n$.
\end{lemma}

\begin{proof}
Recall that the adjoint action is given by $b.r = brb^{-1}$, where $b\in P_{\alpha}$ and $r \in \fp_{\alpha}$. For $1\leq k\leq \ell$,  
define $\pr_k \colon \fp_{\alpha} \to \fp_{\alpha_k}$, where $\pr_k(r) = r_k$, the projection of $r$ onto its $k$-th diagonal block.
Denote $\mcI := \{(k,\iota) \mid 1 \leq k \leq \ell \text{ and } 1 \leq \iota \leq \alpha_k\}$. 
For $(k,\iota) \in \mcI$, define 
$f_{k,\iota} \colon \fp_{\alpha_k} \to \CC$ as 
$f_{k,\iota}(r_k)=\tr(r_k^{\iota})$. 
Define $g_{k,\iota} := f_{k,\iota}\circ \pr_k$, and note $g_{k,\iota} \in \CC[\fp_{\alpha}]$. 
Since
\[
b.g_{k,\iota}(r) = g_{k,\iota}(b^{-1}.r) = g_{k,\iota}(b^{-1}rb)
= f_{k,\iota}(b_k^{-1} r_k b_k) 
= \tr\bigl( (b_k^{-1}r_kb_k)^{\iota} \bigr)
= \tr(b_k^{-1}r_k^{\iota}b_k)
=\tr(r_k^{\iota}),
\] 
where $b_k$ is the $k$-th diagonal block of $b\in P_{\alpha}$, 
the polynomials $\{ g_{k,\iota}\}_{(k,\iota) \in \mcI}$ are algebraically independent (see~\cite[Sec.~2.4]{Kraft-Procesi}) and satisfy $\CC[\{ g_{k,\iota}\}_{(k,\iota) \in \mcI} ]\subseteq \CC[\fp_{\alpha}]^{P_{\alpha}}$. 

Now, let $\lambda \colon \CC^* \to P_{\alpha}$ be a $1$-parameter subgroup defined as 
\begin{align*}
\lambda(t) & = \diag(\underbrace{t^{\ell-1}, \dotsc, t^{\ell-1}}_{\alpha_1 \text{ times}}, \underbrace{t^{\ell-2}, \dotsc, t^{\ell-2}}_{\alpha_2 \text{ times}}, \dotsc, \underbrace{t, \dotsc, t}_{\alpha_{\ell-1} \text{ times}}, \underbrace{1, \dotsc, 1}_{\alpha_{\ell} \text{ times}})
\\ & = \left(  
\begin{array}{c;{1pt/1pt}c;{1pt/1pt}c;{1pt/1pt}c;{1pt/1pt}c} 
t^{\ell-1} \I_{\alpha_1} & &  & & \\ 
\hdashline[1pt/1pt]
& t^{\ell-2} \I_{\alpha_2} &  & &  \\ 
\hdashline[1pt/1pt]
& & t^{\ell-3} \I_{\alpha_3} & & \\ 
\hdashline[1pt/1pt]
& & & \ddots & \\ 
\hdashline[1pt/1pt]
& & & & \I_{\alpha_{\ell}} \\ 
\end{array}
\right).
\end{align*}
Denote the block matrix version of $r$ with block size vector $\alpha$ as
\[
r = 
\left(  
\begin{array}{c;{1pt/1pt}c;{1pt/1pt}c;{1pt/1pt}c;{1pt/1pt}c} 
\bfr_{11} & \bfr_{12} & \bfr_{13} & \ldots& \bfr_{1\ell} \\ 
\hdashline[1pt/1pt]
0& \bfr_{22} & \bfr_{23} &\ddots & \bfr_{2\ell} \\ 
\hdashline[1pt/1pt]
0& 0& \bfr_{33} &\ddots & \bfr_{3\ell}\\ 
\hdashline[1pt/1pt]
\vdots  &\vdots  & \vdots & \ddots & \vdots\\ 
\hdashline[1pt/1pt]
0 & 0& 0& \ldots & \bfr_{\ell\ell} \\ 
\end{array}
\right),  
\]
then 
\[ 
\small 
\lambda(t).r
= \lambda(t) \cdot r \cdot \lambda(t)^{-1}
= 
\left(  
\begin{array}{c;{1pt/1pt}c;{1pt/1pt}c;{1pt/1pt}c;{1pt/1pt}c} 
\bfr_{11} & \bfr_{12} t & \bfr_{13} t^2 & \ldots& \bfr_{1\ell} t^{\ell-1}\\ 
\hdashline[1pt/1pt]
0& \bfr_{22} & \bfr_{23} t &\ddots & \bfr_{2\ell} t^{\ell-2} \\ 
\hdashline[1pt/1pt]
0& 0& \bfr_{33} &\ddots & \bfr_{3\ell} t^{\ell-3} \\ 
\hdashline[1pt/1pt]
\vdots  &\vdots  & \vdots & \ddots & \vdots\\ 
\hdashline[1pt/1pt]
0 & 0& 0& \ldots & \bfr_{\ell\ell} \\ 
\end{array}
\right),  
\] 
where $(r_{ij})_{\alpha_k\times \alpha_k}$ is the $k$-th diagonal block of $r$. 

Now, since 
\[ 
\small 
\lim_{t\rightarrow 0} \lambda(t).r = 
\left(  
\begin{array}{c;{1pt/1pt}c;{1pt/1pt}c;{1pt/1pt}c} 
\bfr_{11} & 0 &\ldots & 0\\ 
\hdashline[1pt/1pt]
 0& \bfr_{22} &\ddots &0 \\ 
\hdashline[1pt/1pt]
\vdots  &\vdots  & \ddots & \vdots\\ 
\hdashline[1pt/1pt]
0 & 0& \ldots & \bfr_{\ell\ell} \\ 
 \end{array}
\right),  
\] 
the $P_{\alpha}$-invariant polynomials are independent of the coordinate functions in the nilradical. Thus $\CC[\mathfrak{p}_{\alpha}]^{P_{\alpha}} \cong \CC[\{ \tr(r_k^{\iota})\}_{(k,\iota) \in \mcI}]$.  
\end{proof}

\begin{example}
Let us examine when $P = B$.
For the adjoint (resp.\ coadjoint) action of $B$ on $\fb$,
 we have $\fb /\!\!/ B \cong \Spec(\CC[\diag(r)]) \cong \CC^n$
by Lemma~\ref{lem:coadjoint-parabolic}.  
\end{example}

\section{The Jordan {\texorpdfstring{$P$}{P}}-semicanonical form}
\label{section:Jordan-P-form}

In this section, we will construct an analog of the Jordan canonical form for a matrix under the $P$-action.

Let us first consider the $n = 2$ case with $P = B$. So take
\[
r = \begin{pmatrix} r_{11} & r_{12} \\ 0 & r_{22} \end{pmatrix}\in \mathfrak{b}
\qquad \mbox{ and }\qquad
b = \begin{pmatrix} b_{11} & b_{12} \\ 0 & b_{22} \end{pmatrix}\in B. 
\]
Then we have
\begin{equation}
\label{eq:conj_rank2}
b r b^{-1} = \begin{pmatrix}
r_{11} & \frac{b_{12} (r_{22} - r_{11}) + b_{11} r_{12}}{b_{22}} \\
0 & r_{22}
\end{pmatrix}.
\end{equation}
We see that the diagonal entries $\{r_{11}, r_{22}\}$ are fixed under the $B$-conjugation action. Now, let us examine the upper right entry, and see how we can obtain
\[
b_{12} (r_{22} - r_{11}) + b_{11} r_{12} = 0.
\]
Note that we have multiplied the upper right entry by $b_{22}$ (recall that $b_{22} \neq 0$ because $b_{22}$ is a diagonal entry of $b\in B$). 
Therefore, if $r_{22} \neq r_{11}$, we can take $b_{12} =-  b_{11} r_{12} / (r_{22} - r_{11})$ (note we can have $b_{12} = 0$ since it is an off-diagonal entry). 
If $r_{11}=r_{22}$, then we must have $r_{12} = 0$ since $b_{11} \neq 0$. Therefore, we can classify the orbits $\orbit{r} = B r B^{-1}$ into the following 3 types: 
\begin{description}
\item[Distinct eigenvalues] $r_{11} \neq r_{22}$;
\item[Two Jordan blocks] $r_{11} = r_{22}$ and $r_{12} = 0$;
\item[One Jordan block] $r_{11} = r_{22}$ and $r_{12} \neq 0$.
\end{description}

To generalize this to arbitrary $n$, we need to weaken the classical notion of Jordan canonical form for $P$-orbits. We construct an embedding $\xi_{\bfa} \colon \mathfrak{gl}_m \to \mathfrak{gl}_n$ 
by considering an (ordered) subset $\bfa = \{ a_1 < a_2 < \dotsc < a_m \} \subseteq \{1,2, \dotsc, n\}$ and mapping a matrix $(x_{pq})_{p,q=1}^m$ to $(x'_{uv})_{u,v=1}^n$ by
\[
x'_{uv} = \begin{cases}
x_{pq} & \text{if $a_p = u$ and $a_q = v$ for some $1 \leq p \leq q \leq n$}, \\
1 & \text{if $u = v$ and there does not exists $p$ such that $a_p = u$}, \\
0 & \text{otherwise}.
\end{cases}
\] 
With a slight abuse of a notation, we will also write $\xi_{\bfa} $ for the embedding 
$\xi_{\bfa} \colon \GL_m \to \GL_n$  of Lie groups, 
which naturally restricts to an embedding of a parabolic subgroup of $\GL_m$ into a parabolic subgroup of $\GL_n$.

\begin{remark}
The embedding $\xi_{\bfa}$ corresponds to embedding $\GL_m$ along the roots
\[
\alpha_{a_1,a_2}, \alpha_{a_2,a_3}, \dotsc, \alpha_{a_{m-1},a_m},
\]
where $\alpha_{p,q} = \alpha_p + \alpha_{p+1} + \ldots + \alpha_q$.
Furthermore, the order of the eigenvalues can only change within a particular embedded $\GL_m$ block.
\end{remark}

We say that a matrix $M \in \fp$ is in \defn{Jordan $P$-semicanonical form} if there exists a (set) partition $\bfa^{(1)} \sqcup \bfa^{(2)} \sqcup \ldots \sqcup \bfa^{(k)} = \{1, 2, \ldots, n\}$ (recall that in a partition, we have $\bfa^{(j)} \neq \emptyset$ for all $j$) such that the following holds:
\begin{enumerate}
\item \label{cond:separate} if $j \neq j'$, then $M_{aa'} = 0$ for all $a \in \bfa^{(j)}$ and $a' \in \bfa^{(j')}$,
\item \label{cond:blocks} the diagonal blocks of $M$ are in Jordan canonical form, and
\item \label{cond:distinct} if $M_{aa} \neq M_{a'a'}$ for $a \in \bfa^{(j)}$ and $a' \in \bfa^{(j')}$, then $j \neq j'$.
\end{enumerate}
We call the set $\bfa^{(j)}$ a \defn{Jordan $P$-block} or simply a \defn{$P$-block} (or \defn{block} when $P$ is clear).

\begin{theorem}
\label{thm:Jordan_semicanonical_form}
For any block upper triangular matrix $r\in \fp$, there exists at least one matrix in Jordan $P$-semicanonical form in its orbit $\{ b r b^{-1} \mid b \in P \}$.
\end{theorem}

\begin{proof}
For $p < q$, define the matrix
\[
b_{pq}(x) := \xi_{\{p,q\}}\left( \begin{pmatrix} 1 & x \\ & 1 \end{pmatrix} \right) = \begin{pmatrix}
1 \\
& \ddots & & x \\
& & \ddots \\
& & & \ddots \\
& & & & 1
\end{pmatrix},
\]
where the entry for $x$ is in position $(p, q)$.

We give an algorithm to construct a Jordan $P$-semicanonical form of $r$.
Suppose $P$ has block size vector $\alpha = (\alpha_1, \alpha_2, \dotsc, \alpha_{\ell})$, and so it has a Levi subgroup of $L = \GL_{\alpha_1} \times \GL_{\alpha_2} \times \ldots \times \GL_{\alpha_{\ell}}$.
Suppose $r$ has distinct eigenvalues $\lambda^{(1)}, \dotsc, \lambda^{(k)}$.
First, note that $r$ is a block upper triangular matrix, and we can act by $L$ on each factor so that each diagonal block is in Jordan canonical form.
Thus our resulting matrix satisfies (\ref{cond:blocks}); in particular, we can now assume $r$ is upper triangular.

Next, consider the partition $\bfa^{(1)} \sqcup \bfa^{(2)} \sqcup \ldots \sqcup \bfa^{(k)} = \{1, 2, \ldots, n\}$ given by $\bfa^{(j)} = \{ a \mid r_{aa} = \lambda^{(j)} \}$.
Therefore, we have (\ref{cond:distinct}).
Lastly, our algorithm proceeds through entries in $r$ in the order of
\[
(n-1,n), (n-2,n-1), (n-2, n), \dotsc, (1,2), (1,3), \dotsc, (1,n)
\]
(\textit{i.e.}, we proceed row-by-row from left-to-right, bottom-to-top).
Suppose we are at step $(a, a')$.
If $r_{aa} = r_{a'a'}$, then we continue to the next step.
Otherwise we have $r_{aa} \neq r_{a'a'}$ (\textit{i.e.}, different eigenvalues), and we construct a new matrix
\[
r' := b_{aa'}(x) \cdot r \cdot b_{aa'}(x)^{-1}
\]
that we use in the next step.
To show~(\ref{cond:separate}), it is sufficient to show (by induction) $r'_{aa'} = 0$ and this conjugation does not change any previously set entry, nor any entry in $L$.
A direct computation shows that
\begin{equation}
\label{eq:U2_computation}
r'_{pq} = \begin{cases}
x (r_{a'a'} - r_{aa}) + r_{aa'} & \text{if } p = a, q = a', \\
r_{aq} + x r_{a'q} & \text{if } p = a, q > a', \\
r_{pa'} - x r_{pa} & \text{if } p < a, q = a', \\
r_{pq} & \text{otherwise}.
\end{cases}
\end{equation}
Thus, it is straightforward to see that $r'_{aa'} = 0$ and we do not change any previously set entry, nor any entry in $L$ (note that there is a non-zero entry at $(a,a')$ in $L$ if and only if $r_{aa} = r_{a'a'}$ as it is an off-diagonal entry in a Jordan block).
\end{proof}

\begin{example}
Let $P = B$.
It is possible that a $B$-orbit can contain multiple Jordan $B$-semicanonical forms. For example, both of the following matrices are in Jordan $B$-semicanonical form consisting of a single $B$-block:
\[
M = \begin{pmatrix}
0 & 1& 1& 0 \\ 
0 & 0& 0& 0 \\ 
0 & 0& 0& 1 \\ 
0 & 0& 0& 0 \\ 
\end{pmatrix},
\qquad\qquad
M' = \begin{pmatrix}
0 & 1& 0& 0 \\ 
0 & 0& 0& 1 \\ 
0 & 0& 0& 1 \\ 
0 & 0& 0& 0 \\ 
\end{pmatrix},
\]
but are in the same $B$-orbit as one can obtain $M'$ from $M$ by conjugation using 
\[
\begin{pmatrix}
1 & 0& 0& 0 \\ 
0 & 1& 1& 0 \\ 
0 & 0& 1& 0 \\ 
0 & 0& 0& 1 \\ 
\end{pmatrix} \in B.
\]
\end{example}

\section{Proof of Theorem~\ref{theorem:Grothendieck--Springer-resolution-complete-intersection}}
\label{section:proof-gs-resoln-ci}

Our proof of the complete intersection of the irreducible components of the preimage of $0$ under the Borel moment map closely follows the proof of~\cite[Thm.~1.1]{GG06}. 

Throughout this section, we assume that $\fp = \fp_{\alpha}$ has at most $5$ blocks, that is to say we fix an $\alpha = (\alpha_1, \alpha_2, \dotsc, \alpha_{\ell})$ with $\ell \leq 5$.
We note that Theorem~\ref{thm:Hille-Roehrle-block-vec} and the Jordan $P$-semicanonical form will assume the role of~\cite[Lemma~2.1]{GG06}.
Next, we note that
\[
d_{\fp} := \dim \fp = \binom{n+1}{2} + \sum_{k=1}^{\ell} \binom{\alpha_k}{2}.
\]

Let $\mcO_r$ be the $P$-orbit of $r$ in $\fp$, and define the reduced subscheme $\fM(\mcO_r) := \{ (r',s',i',j') \in \fM \mid 
r' \in \mcO_r \}$. 
We give the analog of~\cite[Prop.~2.4]{GG06}.

\begin{proposition}
\label{prop:Lagrangian_subscheme}
Given a $P$-orbit $\mcO_r \subseteq \fp$ of $r\in \fp$, the subscheme $\fM(\mcO_r)$ in $T^*\fP$ is Lagrangian.
\end{proposition}

\begin{proof}
Assume that $\fP$ is an arbitrary smooth $P$-variety. Let $\Orb$ denote the set of all $P$-orbits of $\fM$, and for each orbit $\mcO \in \Orb$, let $T^*_{\mcO} \fP \subseteq T^*\fP$ denote the conormal bundle to $\mcO$. The natural $P$-action on $T^*\fP$ is Hamiltonian with moment map $\mu \colon T^*\fP \to \fp^*$,  
and we have
\begin{equation}\label{eqn:preimage-zero-conormal-bundle-orbits}
\mu^{-1}(0) = \bigcup_{\mcO \in \Orb} T^*_{\mcO} \fP
\end{equation}
since the conormal bundle to $\mcO$ is a subbundle of $T^*\fP|_{\mcO}$ consisting of cotangent vectors to $\fP$ that are zero on the cotangent bundle to $\mcO$.  Since $\mcO \subseteq \fM$, it follows that  $T_{\mcO}^*\fP \subseteq \mu^{-1}(0)$. 
Since $\mcO \in \Orb$, Equation~\eqref{eqn:preimage-zero-conormal-bundle-orbits} follows. 

Now, assume that $\fP = \fp \times \CC^n$.  
Since we can write $\mu^{-1}(0)$ as a union of conormal bundles to $\mcO$ as in~\eqref{eqn:preimage-zero-conormal-bundle-orbits}, 
given any point $(r',s',i',j')\in \fM(\orbit{r})$ and any pair of tangent vectors $X,Y$ in the tangent space $T_{(r',s',i',j')}\fM(\mcO)$, 
we have $\omega_{(r',s',i',j')}(X,Y) = 0$. That is, given a canonical symplectic form $\omega= \sum_{u \geq v} ds_{uv} \wedge dr_{vu}$, 
the restricted symplectic form 
$\omega|_{\fM(\mcO)}$  vanishes. 
Finally, given the projection map $\pi \colon \fp \times (\CC^n)^* \to \fp$ onto the first factor and a conjugacy class $\mcO_r \subseteq \fp$, the set $\pi^{-1}(\mcO_r)$ is a finite union of $P$-orbits from Theorem~\ref{thm:Hille-Roehrle-block-vec}. 
So $\fM(\orbit{r})$ is a finite union of $T^*_{\mcO} \fP$, where $\mcO\in \Orb$. 
Let us enumerate these orbits as $\mcO_{\eta}\in \Orb$, where $1\leq \eta \leq \ell$. 
Since $\dim \fM(\orbit{r}) 
= \dim T^*_{\orbit{\eta}} \fP    
= \frac{1}{2} \dim T^*\fP$, it follows that $\fM(\orbit{r})$ is Lagrangian.   
\end{proof}

Next, we have an analog of~\cite[Prop.~2.5]{GG06}, 
working with 
$\fp /\!\!/ P \cong \CC^{(\alpha)}$ (Lemma~\ref{lem:coadjoint-parabolic}).
For a matrix $r$, we let $\Spec(r)$ be the tuple of diagonal entries of any Jordan $P$-semicanonical form of $r$, which we then consider as an element in $\CC^{(\alpha)}$. In other words, $\Spec(r)$ is equivalent to the tuple of unordered tuples of eigenvalues (counted with multiplicities) of each $P_{\alpha}$-block of any Jordan $P$-semicanonical form.
We note from the properties of a Jordan $P$-semicanonical form that this is well-defined.

\begin{proposition}
\label{prop:flat_morphism}
Let $\alpha$ have length at most $5$.
Consider the map $\Lambda \colon T^*\fP \to \fp /\!\!/ P \cong \CC^{(\alpha)}$ given by
\[
(r,s,i,j) \mapsto \Spec(r).
\]
Then the morphism
\[
\mu \times \Lambda \colon T^*\fP \to \fp^* \times (\fp /\!\!/ P) 
\qquad \mbox{ given by }
\qquad 
(r,s,i,j)\mapsto ([r,s]+\overline{ij}, \: \Spec(r))
\]
is flat. Moreover, all nonempty (scheme-theoretic) fibers of this morphism have dimension $d_{\fp} + n$.
\end{proposition}
 
\begin{proof}
For any tuple $\vec{r} = (r_1, r_2, \dotsc, r_n) \in \CC^{(\alpha)}$, the set of all elements $r \in \fp$ such that 
$\Lambda(r,s,i,j) = \vec{r}$ (for fixed $s$, $i$, and $j$) is a finite union of $P$-conjugacy classes by considering the possible Jordan semicanonical $P$-block decompositions of $r$ and Theorem~\ref{thm:Hille-Roehrle-block-vec} implies there is only a finite number of $P$-orbits for any such block decomposition. Therefore, the zero fiber of $\mu \times \Lambda$, denoted by $\xi$, is equal to a finite union of Lagrangian subschemes $\fM(\mcO_r)$ from Proposition~\ref{prop:Lagrangian_subscheme}, where $\mcO_r$ is a conjugacy class of a nilpotent matrix $r\in \fp$. Hence, we have
\[
\dim \xi \leq \frac{1}{2} \dim T^* \fP = d_{\fp} + n = \dim T^* \fP - \dim \fp \times \CC^{(\alpha)}
\]
since Lagrangian subschemes can have dimension at most $\frac{1}{2} \dim T^* \fP$.

We define a $\CC^*$-action on $T^*\fP$ by scalar multiplication, \textit{i.e.,}\ we have~$\alpha.(r,s,i,j) = (\alpha r, \alpha s, \alpha i, \alpha j)$. 
Next let $\CC^*$ act on $\fp^* \times \CC^{(\alpha)}$ by $\zeta.(s, i) := (\zeta^2 s, \zeta i)$. 
Then the map $\mu \times \Lambda$ is a $\CC^*$-equivariant morphism since
\begin{align*}
(\mu \times \Lambda) \bigl( \zeta.(r,s,i,j) \bigr) & = (\mu \times \Lambda)(\zeta r, \zeta s, \zeta i, \zeta j)
\\ & = \bigl([\zeta r, \zeta s] + (\zeta i)(\zeta j), \Spec(\zeta r) \bigr)
\\ & = \bigl( \zeta^2 ([r,s] + ij), \zeta \Spec(r) \bigr)
\\ & = \zeta.\bigl([r,s] + ij, (r_{11}, \dotsc, r_{nn}) \bigr) = \zeta.(\mu \times \Lambda) (r,s,i,j). 
\end{align*}  
Let $\CC[T^*\fP]_{\leq i}$ be the set of polynomials of degree less than or equal to $i$. 
This forms an $\NN$-filtration on $\CC[T^*\fP]= \bigcup_{i=0}^{
\infty} \CC[T^*\fP]_{\leq i}$.   
Let $\widehat{\CC[T^*\fP]}$ be the set of all finite expressions of the form $\sum_i b_i \zeta^i$ with $b_i \in \CC[T^*\fP]_{\leq i}$.  
Since $\CC^*$-equivariance is well-defined under the limit as $\alpha \to 0$, 
we have a natural ring embedding $\CC[\zeta] \hookrightarrow \widehat{\CC[T^*\fP]}  
\iso \sum_{i \in \NN} \CC[T^*\fP]_{\leq i} \zeta^i \subseteq \CC[T^*\fP][\zeta]$,   
which gives a surjective morphism of algebraic varieties $f \colon \Spec(\widehat{\CC[T^*\fP]}) \twoheadrightarrow \CC$.

Since $\zeta$ is not a zero-divisor, $\widehat{\CC[T^*\fP]}$ is a torsion-free $\CC[\zeta]$-module. So 
$\widehat{\CC[T^*\fP]}$ is flat. Thus $\mu \times \Lambda$ is a flat morphism, 
and the dimension of any fiber of the map $\mu \times \Lambda$ is $d_{\fp} + n$ (also see~\cite[Sec.~2.3.9]{MR2838836}).
\end{proof}

Finally, we obtain analogs of~\cite[Cor.~2.6]{GG06} and~\cite[Cor.~2.7]{GG06}. Let $\overline{\Lambda} := \Lambda|_{\fM}$ be the restriction of $\Lambda$ to the closed subscheme $\fM$.

\begin{corollary}
The moment map $\mu$ is flat.
\end{corollary}

\begin{proof}
The moment map $\mu$ is the projection of $\mu\times \Lambda$ onto the first factor $\pr_1 \colon \fp^*\times \CC^n \to \fp^*$. Since the composition $\pr_1 \circ \; (\mu\times \Lambda)$ is flat, $\mu$ is flat. 
\end{proof}

\begin{corollary}
\label{cor:complete_intersection}
The scheme $\fM$ is a complete intersection in $T^*\fP$ with $\dim \fM = d_{\fp} + 2n$. Moreover, the map $\overline{\Lambda} \colon \fM \to \CC^{(\alpha)}$ from Proposition~\ref{prop:flat_morphism} is a flat morphism with fibers of dimension $d_{\fp} + n$ that are Lagrangian subschemes in $T^*\fP$.
\end{corollary}

\begin{proof}
The restriction $\overline{\Lambda}$ is surjective since clearly $(r,0,0,0) \in \mu^{-1}(0)$ for any $r \in \fp$. Thus all fibers of $\overline{\Lambda}$ are nonempty.

Taking a flat base change with respect to the embedding $\{0\} \times \CC^n \to \fp^* \times \CC^n$ yields that the scheme $\fM = \mu^{-1}(0) = (\mu \times \Lambda)^{-1}(\{0\} \times \CC^n)$ is a complete intersection in $T^*\fP$ and that $\overline{\Lambda} \colon \fM \to \CC^n$ is flat. This implies that the dimension of any irreducible component of any fiber of this morphism $\overline{\Lambda}$ is $\dim T^* \fP - \dim \fp \times \CC^n = d_{\fp} + n$.

Consider a fiber $\xi = \overline{\Lambda}^{-1}(\Spec(r))$, and so for $(r,s,i,j) \in \xi$, the diagonal entries of $r$ are $\Spec(r)$. From Theorem~\ref{thm:Hille-Roehrle-block-vec}, there exists only finitely many conjugacy classes $\mcO_r$ in $\fp$, and so $\xi$ is a finite union of $\fM(\mcO_r)$. By Proposition~\ref{prop:Lagrangian_subscheme}, $\fM(\mcO_{r})$ is a Lagrangian subscheme. Next, any irreducible component of the corresponding scheme-theoretic fiber must be the closure of an irreducible component of $\fM(\mcO_r)$ for some conjugacy class $\mcO_r$ of $\fp$, and hence a Lagrangian subscheme, since every irreducible component of the scheme-theoretic fiber has dimension $\dim \fM(\mcO_r)$.
\end{proof}

We conclude that Corollary~\ref{cor:complete_intersection} yields Theorem~\ref{theorem:Grothendieck--Springer-resolution-complete-intersection}.

\section{Proof of Theorem~\ref{thm:irred-comp-itemized}}
\label{section:proof-irred-comp-itemized}

The results in this section are in~\cite{Nevins-GSresolutions} (with assuming Theorem~\ref{theorem:Grothendieck--Springer-resolution-complete-intersection}). We give a slightly different proof that instead closely follows~\cite{GG06} rather than lifting up to the $\fg$ setting.

For this section, we fix an $\alpha = (\alpha_1, \alpha_2, \dotsc, \alpha_{\ell})$ with $\ell \leq 5$ (as in the previous section). We will also consider $\alpha \in \NN^{\ell}$.
Let $e_{a}$ denote the unit vector with $1$ in the $a$-th entry and $0$ otherwise. 

We give the analog of~\cite[Lemma~2.8]{GG06}. Note that we have to give an (block) upper triangular version of~\cite[Lemma~2.8]{GG06}, so we have interchanged the roles of $i \leftrightarrow j$. This is~\cite[Eq.~(4.1)]{Nevins-GSresolutions} when $P = B$. Since this does not require the lift from~\cite[Lemma~4.1]{Nevins-GSresolutions}, we do not require~\cite[Lemma~2.3]{GG06}. Thus, in practice, ours is distinct as we do not lift to the $\fg$ and $G$ setting and instead work directly with the Borel subalgebra and the Borel subgroup. 

Define a partial order on $\NN^{\ell}$ by $u \leq v$ if $u_k \leq v_k$ for all $k$.
For an element $\bfa = (a_1, a_2, \dotsc, a_{\ell}) \in \NN^{\ell}$ such that $\bfa \leq \alpha$, define the sets
\begin{subequations}
\label{eq:updown_sets}
\begin{align}
\begin{split} \bfa_{\downarrow} & := \{1, 2, \dotsc, a_1\} \cup
  \{\alpha_1+1, \alpha_1+2, \dotsc, \alpha_1+ a_2\}
\\ & \hspace{20pt}  \cup \dotsc \cup \{\alpha_1 + \ldots + \alpha_{\ell-1} + 1, \dotsc, \alpha_1 + \ldots + \alpha_{\ell-1} + a_{\ell}\}, \end{split}
\\  \begin{split} \bfa_{\uparrow} & := \{\alpha_1-a_1+1, \alpha_1-a_1+2, \dotsc, \alpha_1\}
\\ & \hspace{20pt} \cup \{\alpha_1+\alpha_2-a_2+1, \alpha_1+\alpha_2-a_2+2, \dotsc, \alpha_1+\alpha_2\}
\\ & \hspace{20pt}  \cup \dotsc \cup \{n - a_{\ell} + 1, n - a_{\ell} + 2, \dotsc, n\}. \end{split}
\end{align}
\end{subequations}
Note that $\bfa_{\downarrow} \sqcup (\alpha - \bfa)_{\uparrow} = \{1, \dotsc, n\}$.

\begin{lemma}
\label{lemma:rep_distinct_eigenvalues}
Let $(\check{r},\check{s},\check{\imath},\check{\jmath}) \in \fM$ such that $\rho_p \neq \rho_q$ for all $1 \leq p < q \leq n$, where $(\rho_p)_{p=1}^n$ are the eigenvalues of $\check{r}$ (\textit{i.e.}, the eigenvalues of $\check{r}$ are pairwise distinct). Then the $P$-orbit of $(\check{r},\check{s},\check{\imath},\check{\jmath})$ contains a representative $(r,s,i,j)$ such that
\begin{enumerate}
\item\label{item:rep-dist-ev-1}  
$r = \diag(\rho_1, \rho_2, \dotsc, \rho_n)$ is diagonal, 
\item\label{item:rep-dist-ev-2} $i = \sum_{a \in \bfa_{\downarrow}} e_{a}$ and $j = \sum_{a' \in \bfa'_{\uparrow}} e_{a'}$ for some $\bfa, \bfa' \in \NN^{\ell}$ such that $\bfa + \bfa' \leq \alpha$, 
\item\label{item:rep-dist-ev-3} $s = (s_{pq})_{p,q=1}^n$ is given by
\[
s_{pq} = \begin{cases}
\quad\: \: \sigma_p & \text{if } p = q, \\
{\displaystyle \frac{1}{\rho_{p} - \rho_{q}}} & \text{if } p \in \bfa_{\downarrow}, q \in \bfa'_{\uparrow}, \text{ and } p > q, \\
\quad\:\: 0 & \text{otherwise},
\end{cases}
\]
for some $\sigma_1, \dotsc, \sigma_n \in \CC$.
\end{enumerate}
Conversely, given $(r,s,i,j)$ that satisfy these conditions for any choice of $\rho_1, \dotsc, \rho_n$, $\sigma_1, \dotsc, \sigma_n$, $\bfa'$, $\bfa$ (\textit{i.e.}, with $\rho_1, \dotsc, \rho_n$ pairwise distinct and $\bfa + \bfa' \leq \alpha$), we have $(r,s,i,j) \in \fM$ with $\supp(i) = \bfa_{\downarrow}$, and $\supp(j) = \bfa'_{\uparrow}$.
\end{lemma}

Note that $s_{pp} = \sigma_p$ in Lemma~\ref{lemma:rep_distinct_eigenvalues}, but $\sigma_p$  are considered as free variables for the converse. 

\begin{proof}
From Theorem~\ref{thm:Jordan_semicanonical_form}, we can assume $r$ is diagonal, so we obtain~\eqref{item:rep-dist-ev-1}.
For simplicity, we consider $P = B$, but the general case is similar (note that we can permute the entries of $i$ and $j$ within a block $\GL_{\alpha_k} \subseteq P$ and $\bfa_{\downarrow} \cap \bfa'_{\uparrow} = \emptyset$). 

Next, we need to solve
\begin{equation}
\label{eq:zero_fiber_repr}
[r,s] + ij + \fu =
\begin{pmatrix}
0 & & & \\
(r_{22} - r_{11}) s_{21} & 0 & & \\
\vdots & \ddots & \ddots & & \\
 (r_{nn} - r_{11}) s_{n1} & \ldots & (r_{nn} - r_{n-1,n-1}) s_{n,n-1} & 0
\end{pmatrix}
+ ij + \fu = 0 + \fu.
\end{equation}
Let $\bfa = \supp(i)$ and $\bfa' = \supp(j)$.
Note that $ij$ is the matrix with a row $p$ given by $j$ scaled by $i_j$ if $p \in \bfa$ and is $0$ otherwise.
However, to have a solution to~\eqref{eq:zero_fiber_repr}, we must have $\bfa \cap \bfa' = \emptyset$ as otherwise there will be a nonzero entry along the diagonal of $ij$.

Next, by computing $br - rb = 0$, we see that the centralizer of $r$ is given by the diagonal matrices. Therefore, we can assume $i$ is a vector with entries $\{0,1\}$. Since $\supp(i) \cap \supp(j) = \emptyset$, we can (independently) scale all nonzero entries of $j$ to be $1$.
Since we are working in $\fg / \fu$, we do not care about the strictly block upper diagonal portion in~\eqref{eq:zero_fiber_repr}. Therefore, we have~\eqref{item:rep-dist-ev-2}. Next, solving for $(s_{pq})_{p,q=1}^n$ yields~\eqref{item:rep-dist-ev-3}. Note that we have also shown the converse statement of the lemma.
\end{proof}

We give the analog of~\cite[Lemma~2.9]{GG06}, which we split into the following two lemmas.
We note that the condition that $\supp(i)$ and $\supp(j)$ are disjoint is redundant from~\eqref{eq:zero_fiber_repr} as noted above, but we have included it for clarity.

Let $\Delta := \{(x_1,\ldots, x_n)\in \CC^n : x_i=x_j\mbox{ for some }i\not=j \}$ denote the big diagonal in $\CC^n$.
For some $\bfa = (a_1, \dotsc, a_{\ell}) \in \NN^{\ell}$ such that $\bfa \leq \alpha$, define the (parabolic) subgroup
\[
S_{\bfa} := S_{a_1} \times S_{\alpha_1-a_1} \times S_{a_2} \times S_{\alpha_2-a_2} \times \ldots \times S_{a_{\ell}} \times S_{\alpha_{\ell}-a_{\ell}} \subseteq S_n,
\]
where the first factor $S_{a_1}$ acts on the first $a_1$ elements, then the second factor $S_{\alpha_1-a_1}$ acts on the next $\alpha_1 - a_1$ elements, etc.

Let $\fD$ be the elements $(r,s,i,j) \in \fM$ such that
\begin{enumerate}
\item $r = (r_{pq})_{p,q=1}^n$ is a diagonal matrix with pairwise distinct eigenvalues $\rho_1, \dotsc, \rho_n$ (note $r_{pp} = \rho_p$ for all $1 \leq p \leq n$), 
\item $i = \sum_{a \in \bfa_{\downarrow}} e_{a}$ and $j = \sum_{a' \in \bfa'_{\uparrow}} e_{a'}$ for some $\bfa, \bfa' \in \NN^{\ell}$  
such that $\bfa + \bfa' \leq \alpha$,
\item $s = (s_{pq})_{p,q=1}^n$ is given by
\[
s_{pq} = \begin{cases}
\quad \: \: \sigma_p & \text{if } p = q, \\
{\displaystyle \frac{1}{\rho_{p} - \rho_{q}}} & \text{if } p \in \bfa_{\downarrow}, q \in \bfa'_{\uparrow}, \text{ and } p > q, \\
\quad\:\: 0 & \text{otherwise},
\end{cases}
\]
for some $\sigma_1, \dotsc, \sigma_n \in \CC$.
\end{enumerate}

\begin{lemma}
\label{lemma:open_irreducible_components}
Choose some $\bfa \in \NN^{\ell}$ such that $\bfa \leq \alpha$. Recall from Theorem~\ref{thm:irred-comp-itemized} that
\[
\fM'_{\bfa} = \bigcup \left\{ \orbit{(r,s,i,j)} \mid (r,s,i,j) \in \fD \text{ with } \begin{gathered} \supp(i) = \bfa_{\downarrow}, \\ \supp(j) = (\alpha - \bfa)_{\uparrow} \end{gathered} \right\}.
\]
Then $\fM'_{\bfa}$ is connected of $\dim \fM'_{\bfa} = d_{\fp} + 2n$, and both of the actions of $P$ and $\fp$ on $\fM'_{\bfa}$ are free.
\end{lemma}

\begin{proof} 
Let $(r,s,i,j) \in \fM'_{\bfa}$ be a representative from Lemma~\ref{lemma:rep_distinct_eigenvalues} with $\supp(i) = \bfa_{\downarrow}$. To see that the isotropy (or stabilizer) group $\{b \in P \mid b.(r,i,j) = (r,i,j)\}$ is trivial, note that $brb^{-1} = r$ (equivalently $[b, r] = 0$) implies that $b$ is a diagonal matrix. Since $b$ is a diagonal matrix and $bi = i$, we must have that $b$ is the identity matrix. Similarly, the isotropy Lie algebra $\{b \in \fp \mid b.(r,i,j) = 0\}$ is trivial. Therefore, the same holds for $(r,s,i,j)$. Note that the representative $(r,s,i,j)$ is unique up to the natural $S_{\bfa}$-action since we cannot permute the diagonal blocks of $r$ by the $P$-action, nor can we permute the entries of $i$ (and similarly for $j$) by elements of the centralizer of $r$ (recall, this is the group of diagonal matrices). Indeed, if we act by any $b \in P$ that is not in $S_{\bfa}$ nor the centralizer of $r$, then $r$ is no longer a diagonal matrix, and we cannot change the support of $i$ when acting by a diagonal matrix (all of whose diagonal entries must be nonzero). Hence, we obtain
\[
\fM'_{\bfa} \iso P \times_{S_{\bfa}} \bigl( (\CC^n \setminus \Delta) \times \CC^n \bigr),
\]
where $S_{\bfa}$ acts diagonally on $(\CC^n \setminus \Delta) \times \CC^n$ (\textit{i.e.}, $\sigma.(x, y) = (\sigma.x, \sigma.y)$),
and the claim follows.
\end{proof}

\begin{lemma}
\label{lemma:too_small_part}
Choose some $\bfa \in \NN^{\ell}$ such that $\bfa \leq \alpha$. Let
\[
\fM''_{\bfa} := \bigcup \left\{ \orbit{(r,s,i,j)} \mid (r,s,i,j) \in \fD \text{ with } \begin{gathered} \supp(i) \subseteq \bfa_{\downarrow}, \; \bfa_{\downarrow} \cap \supp(j) = \emptyset, \\ \supp(i) \cup \supp(j) \neq \{1, \dotsc, n\} \end{gathered} \right\}.
\]
Then $\dim \fM''_{\bfa} < d_{\fp} + 2n$.
\end{lemma}

\begin{proof}
Let $(r,s,i,j) \in \fM''_{\bfa}$ be a representative from Lemma~\ref{lemma:rep_distinct_eigenvalues} with $\supp(i) \subseteq \bfa_{\downarrow}$. Let $\Sigma = \supp(i) \cup \supp(j)$. Consider the subgroup $B'' \subseteq B$ given by the diagonal matrices $(b_{pp})_{p=1}^n$ with
\[
b_{pp} = \begin{cases}
1 & \text{if } p \in \Sigma, \\
\beta_{a''} & \text{if } a'' \in \{1,\dotsc,n\} \setminus \Sigma,
\end{cases}
\]
where $\beta_{a''}$ are arbitrary complex numbers.
Note that $\{1,\dotsc,n\} \setminus \Sigma \neq \emptyset$ by assumption.
Hence, the subgroup $B''$ is of strictly positive dimension that acts trivially on $(r,s,i,j)$. Hence, we must have $\dim \fM''_{\bfa} <\dim \fM'_{\bfa}= \dim \fM = d_{\fp} + 2n$ (the last equality is by Corollary~\ref{cor:complete_intersection}).
\end{proof}

\begin{proof}[Proof of Theorem~\ref{thm:irred-comp-itemized}]
Since the diagonal entries under the $P$-action cannot change, it is straightforward to see that $\fM$ can be written as   
\begin{equation}
\label{eq:decomposition}
\fM = \left( \bigcup_{\bfa \leq \alpha} \fM'_{\bfa} \right) \sqcup \left( \bigcup_{\bfa \leq \alpha} \fM''_{\bfa} \right) \sqcup \overline{\Lambda}^{-1}(\Delta). 
\end{equation}
Next, note that $\dim \Delta < n$, and thus Corollary~\ref{cor:complete_intersection} implies that $\dim \overline{\Lambda}^{-1}(\Delta) < n + \left( d_{\fp} + n \right) = d_{\fp} + 2n$.
Also from Corollary~\ref{cor:complete_intersection}, we have that $\fM$ is a complete intersection, and so every irreducible component must have dimension $d_{\fp} + 2n$. Hence the closures of $\fM''_{\bfa}$ (from Lemma~\ref{lemma:too_small_part}) and $\pi^{-1}(\Delta)$ cannot be irreducible components. The claim that the closure of $\fM'_{\bfa}$ is an irreducible component follows from Lemma~\ref{lemma:open_irreducible_components}, and thus we have obtained all irreducible components from the decomposition~\eqref{eq:decomposition}.

Since there are no fixed points under the $P$-action on $\fM'_{\bfa}$, 
the map $\mu$ is a submersion at generic points of the scheme $\fM$. 
Thus $\fM$ is generically reduced. We have that $\fM$ is Cohen-Macaulay since it is a complete intersection. We conclude that $\fM$ is reduced (see, \textit{e.g.},~\cite[Ex.~18.9]{MR1322960} or~\cite[Thm.~2.2.11]{MR2838836}).
\end{proof}

\section{Proof of Theorem~\ref{theorem:CM-smooth-irreducible-affine-var}}
\label{section:connections-to-CM-particles}

In this section, we will work with $\widetilde{C}_n$, proving Theorem~\ref{theorem:CM-smooth-irreducible-affine-var}.  Our proofs closely follow those of the analogous statements from~\cite{Wilson98}. We first prove a lemma analogous to \cite[Cor.~1.4]{Wilson98}.

\begin{lemma}\label{lemma:commute-scalar}
Fix $(r,s,i,j)\in \widetilde{C}_n$. Suppose $b \in P$ is a matrix that commutes with both $r$ and $s$. Then $b = \lambda \I_n$ for some $\lambda \in \CC$ (\textit{i.e.}, $b$ is a scalar matrix). 
\end{lemma}

\begin{proof}
Consider some representative $\widetilde{s} \in \fg$ of $s$.
Note that $(r,\widetilde{s},i,j)$ is an element of the variety defined by~\cite[Eq.~(1.1)]{Wilson98}.
For any matrix $b$ that commutes with both $r$ and $\widetilde{s}$, any eigenspace of $b$ is a common invariant subspace for $r$ and $\widetilde{s}$. Therefore any eigenspace of $b$ must be $\CC^n$ by~\cite[Lemma~1.3]{Wilson98}. Since this holds for every representative of $s$, we must have $b = \lambda \I$ for some $\lambda \in \CC$ as desired.
\end{proof}

The following corollary is a generalization of \cite[Cor. 1.5]{Wilson98}. 
\begin{corollary}\label{cor:free-action}
The group $P$ acts freely on $\widetilde{C}_n$. 
\end{corollary}

\begin{proof} 
If $(brb^{-1},bsb^{-1},bi, jb^{-1})= (r,s,i,j)$, for some $b \in P$, we have $br = rb$ and $bs = sb$. Therefore, $b$ commutes with $r$ and $s$, and so $b = \lambda \I_n$ by Lemma~\ref{lemma:commute-scalar}. Since we have $bi = \lambda i = i$, we have $\lambda = 1$ and $b$ the identity matrix.
\end{proof} 

We follow \cite[Prop. 1.7]{Wilson98} to obtain: 
\begin{proposition}\label{prop:differential-surjective}
The differential of $\mu$ is surjective at every point of $\widetilde{C}_n$. 
\end{proposition}

\begin{proof}
The differential of $\mu$ at the point $(r,s,i,j)\in T^*\fP$ is 
\[ 
d\mu(X,Y,a,b) = [r, X] + [Y, s] + ib - aj. 
\]  
The annihilator of the image of this map with respect to the nondegenerate bilinear form $(r,s) \mapsto \tr(rs)$ consists of all matrices $R$ such that 
\[ 
[R,r] = [R,s] = Ri = jR = 0. 
\] 
Thus $d\mu$ is surjective at $(r,s,i,j)$ if and only if $R=0$ is the only solution to these equations. Yet if $(r,s,i,j) \in \widetilde{C}_n$, then by Lemma~\ref{lemma:commute-scalar}, the first two equations imply that $R$ is a scalar. The last two equations show that $R=0$. 
\end{proof}

Since $\widetilde{C}_n = \mu^{-1}(-\I_n)$ for the point $-\I_n \in \fp^*$, it follows from Proposition~\ref{prop:differential-surjective} and the implicit function theorem that $\widetilde{C}_n$ is a smooth subvariety of $T^*\fP$ with every component of dimension $d_{\fp} + 2n$. 
It follows from Corollary~\ref{cor:free-action} that the quotient space $C_n$ is a smooth affine variety of dimension $2n$. 

For the remainder of this section, we aim to prove that $C_n$ is irreducible.
Since $C_n$ is smooth, it is equivalent to proving that it is connected (here, we may work with either the classical complex or Zariski topology).

The following is the analog of~\cite[Prop.~1.10]{Wilson98}.

\begin{lemma}\label{lemma:r-diagonalizable}
Let $(\check{r}, \check{s},\check{\imath}, \check{\jmath}) \in \widetilde{C}_n$ such that $\check{r}$ is diagonalizable. Then the eigenvalues of $\check{r}$ are distinct, and the $P$-orbit of $(\check{r},\check{s},\check{\imath},\check{\jmath})$ contains an element $(r,s,i,j)$ such that 
\begin{enumerate}
\item\label{item:r-diag} $r = \diag(-\rho_1, -\rho_2, \dotsc, -\rho_n)$ is diagonal, 
\item\label{item:ij-equal-1} all of the entries in the vectors $i$ (resp.~$j$) are equal to $-1$ (resp.~$1$), 
\item\label{item:s-CM-matrix} $s$ is a \defn{block lower Calogero--Moser matrix}: it has entries
\[ 
s_{pq} = \begin{cases}
\sigma_p \quad & \text{if } p = q, \\
(\rho_p - \rho_q)^{-1} & \text{if } p > q \text{ or if $p < q$ with $(p,q)$ an entry in a diagonal block}, \\
0 & \text{otherwise},
\end{cases}
\] 
for some $\sigma_p \in \CC^n$.
\end{enumerate}
Moreover, the element $(r,s,i,j)$ is unique up to (simultaneous) permutation of $(\rho_i, \sigma_i)_{i=1}^n$ by $S_{\alpha}$.
\end{lemma}

\begin{proof} 
From the Jordan $P$-semicanonical form (Theorem~\ref{thm:Jordan_semicanonical_form}), we can consider $r$ to be a diagonal matrix with diagonal $(-\rho_1, -\rho_2, \dotsc, -\rho_n)$. Note that the diagonal of $[r,s]$ is $(0, \dotsc, 0)$, and so if we consider the diagonal entries $[r,s]+\I_n = -ij$ in~\eqref{eqn:Wilson}, 
we obtain $-i_p j_p = 1$ for all $1 \leq p \leq n$. Thus, $-ij$ is the matrix with every entry being $1$. For all $p > q$ and $p < q$ with $(p,q)$ being an entry in some diagonal block, the $(p,q)$-entry of $[r,s]+\I_n$ is $s_{pq} (\rho_p - \rho_q)$.\footnote{Recall that $s \in \fb^*$ is a lower triangular matrix.} 
Therefore, if a diagonal entry of $r$ repeats, then for some $p$ and $q$, 
we have $s_{pq} (\rho_p - \rho_q)=0\not=1$, which is a contradiction.  
Thus, we have $\rho_p \neq \rho_q$ for $p \neq q$.
Note that the centralizer (equivalently stabilizer) of $r$ is given by diagonal matrices and elements of $S_{\alpha}$, so we can act by diagonal matrices and elements of $S_{\alpha}$ without changing $r$. Acting by the diagonal matrices, we can obtain $i = (-1, -1, \dotsc, -1)$, and so subsequently we have $j = (1, 1, \dotsc, 1)$. However, by fixing $i$ (and $j$ as in~(\ref{item:ij-equal-1})) we only have the action of $S_{\alpha}$ remaining.
Since $1 = s_{pq} (\rho_p - \rho_q)$ for all $p > q$ and $p < q$ with $(p, q)$ in a diagonal block and $s$ is block lower triangular, we obtain that $s$ is a block lower Calogero--Moser matrix.
\end{proof} 

We remark that since the eigenvalues of $r$ are pairwise distinct if $r$ is diagonalizable, not every matrix $r$ occurs as a part of the quadruple $(r,s,i,j)$ in $\widetilde{C}_n$. 
Furthermore, for any $\rho \in \CC^{(\alpha)}$, there exists a unique element $(r,s,i,j)$, up to the $S_{\alpha}$-action, in the $P$-orbit of $(\check{r},\check{s},\check{\imath},\check{\jmath})$ such that $\spec(r) = \rho$ and that satisfies~(\ref{item:ij-equal-1}) and~(\ref{item:s-CM-matrix}) from Lemma~\ref{lemma:r-diagonalizable}.
Therefore, note that in Lemma~\ref{lemma:r-diagonalizable}, we cannot permute the $\rho_i$ using the parabolic subgroup $P$ other than within a particular block. So the set of points $\bigl( (\rho_1, \dotsc, \rho_n), (\sigma_1, \dotsc, \sigma_n) \bigr) \in (\CC^{(\alpha)} \setminus \Delta) \times \CC^n$ is fixed under the $P$-action. 

Let $\pi \colon C_n \to \CC^{(\alpha)}$ be defined by $(r,s,i,j) \mapsto \Spec(r)$. It is clear that $\pi$ is surjective.
A consequence of Lemma~\ref{lemma:r-diagonalizable} is that
\[
C_n^d := \{(r,s,i,j)\in C_n \mid r \text{ is diagonalizable} \} = \pi^{-1}(\CC^{(\alpha)} \setminus \Delta). 
\]
Therefore the parameters $(\rho, \sigma)$ define an isomorphism $C_n^d \iso (\CC^{(\alpha)} \setminus \Delta) \times \CC^n$, with $\pi$ corresponding to the projection onto the first factor. Thus we have the following analog of \cite[Lemma 1.9]{Wilson98}.

\begin{corollary}
\label{cor:pairwise-distinct-eigenvalues-connected}
We have that $\pi^{-1}(\CC^{(\alpha)} \setminus \Delta)$ is connected. 
\end{corollary}

Next, we need the following lemma, which is similar to \cite[Lemma 1.8]{Wilson98}. 

\begin{lemma}\label{lemma:fibers-dimension-pi}
Let $\alpha$ have length at most $5$.
The fibers of $\pi$ have dimension at most $n$. 
\end{lemma}

In order to prove Lemma~\ref{lemma:fibers-dimension-pi}, we first study the restriction to $\widetilde{C}_n$ of the projection $\pr \colon T^*\fP \to \CC^{2n}$ onto the last two factors $(i,j)$; in other words, we have $\pr(r,s,i,j) = (i,j)$. The following generalizes \cite[Lemma 1.11]{Wilson98}. 

\begin{lemma}\label{lem:r0-fixed-n-dim-subvariety}
Fix some $r_0 \in \fb$, and define
\[
\widetilde{C}_n(r_0) = \{ (r,s,i,j) \in \widetilde{C}_n \mid r = r_0\}.
\]
Then $\pr(\widetilde{C}_n(r_0))$ is contained in an $n$-dimensional subvariety of $\CC^{2n}$. 
\end{lemma}

\begin{proof}
Since $\pr$ is $P$-invariant, assume that $r_0$ is in Jordan $P$-semicanonical form. 

We first consider the case when $r_0$ consists of a single Jordan $P$-block, which implies that within each diagonal block $d_k$ of $r_0$
\[
(d_k)_{pq} = \begin{cases}
\lambda_k & \text{if } p = q, \\
1 & \text{if } p = q+1, \\
0 & \text{otherwise}.
\end{cases}
\]
Recall that for a matrix $A$ to have rank $1$, for the minimal $d$ such that $A_{p-d,p} \neq 0$ for some $p$, there exists a unique such $p$.
In other words, the lowest diagonal that is not $0$ has exactly one nonzero entry.
Next in $X := [r_0, s] + \I_n$, we have within the $k$-th diagonal block of $X$
\[
\sum_{p=1}^{\alpha_k - \lvert d \rvert} X_{s+p-d,s+p} = 0
\]
for all $\alpha_k \leq d < 0$, where $s = \alpha_1 + \ldots + \alpha_{k-1}$.
Then we can take the representative $\overline{X} \in \fp^*$ with all lower diagonal blocks being $0$, hence we have
\[
\sum_{p=1}^{n - \lvert d \rvert} X_{p-d,p} = 0
\]
for all $d < 0$.
Therefore, $X$ is upper triangular since $\rk(\overline{X}) = \rk(-ij) = 1$. Moreover, there exists a unique nonzero entry on the main diagonal of $\overline{X}$, which is equal to
\[
\tr(X) = \tr(\overline{X}) = \tr([r_0,s]) + \tr(\I_n) = 0 + n = n.
\]
Next, suppose the first nonzero entry in $j$ occurs at position $p$, and so the last nonzero entry in $-i$ also occurs in $p$. Note that $-i_p j_p = \tr(\overline{X})$ from above. Therefore, for each of the pairs $(i,j)$, there are $n$ families corresponding to the choice of entry $p$. Every such family has dimension $n$ since there are $n-p+1$ parameters in $j$ and $p$ parameters in $i$ with one relation: $-i_p j_p = n$. Hence, the claim holds when $r_0$ has a single Jordan $P$-block.

Now, assume $r_0 = \bigoplus_{\bfa} r_{\bfa}$, the direct sum of several Jordan $P$-blocks of sizes $n_{\bfa}$, and we write $s_{\bfa\bfb}$, $(i_{\bfa})$, $(j_{\bfa})$ for the corresponding $P$-block decompositions of $s$, $i$, and $j$, respectively. 
Then taking the $(\bfa,\bfa)$-block in~\eqref{eqn:Wilson} gives 
$[r_{\bfa\bfa}, s_{\bfa\bfa}] + \I_{n_{\bfa}} = -i_{\bfa} j_{\bfa}$. 
By the above, there are only at most $n_{\bfa}$ parameters in $(i_{\bfa}, j_{\bfa})$, and so the claim follows.
\end{proof}

Next, we obtain a result analogous to~\cite[Cor. 1.12]{Wilson98}. 

\begin{corollary}\label{cor:fixed-r-orbit}
Fix a conjugacy class $\mcO_r$ in $\fp$, and define
\[
C_n(\mcO_r) := \{ (r',s,i,j) \in C_n \mid r' \in \mcO_r \}.
\]
Then $\dim C_n(\mcO_r) \leq n$.
\end{corollary}

\begin{proof}
Fix $r_0 \in \mcO_r$ such that $C_n(\mcO_r) = \widetilde{C}_n(r_0) / P^{r_0}$ (recall that $P^{r_0}$ denotes the centralizer of $r_0$ in $P$). From~\eqref{eqn:Wilson},  it is clear that the part of $\widetilde{C}_n(r_0)$ lying over a fixed $(i,j)\in \CC^{2n}$ is parameterized by the Lie algebra of $P^{r_0}$. We have $\dim \widetilde{C}_n(r_0) \leq n + \dim P^{r_0}$ by Lemma~\ref{lem:r0-fixed-n-dim-subvariety}. Since the $P^{r_0}$-action on $\widetilde{C}_n(r_0)$ is free by Corollary~\ref{cor:free-action}, the claim follows.
\end{proof}

\begin{proof}[Proof of Lemma~\ref{lemma:fibers-dimension-pi}]
Note that the conjugacy classes of $\fp$ can be parameterized by the conjugacy classes of the Levi and the nilpotent $\fu$ subalgebras.
Recall that for a fixed $\rho \in \CC^{(\alpha)}$ such that $\Spec(r) = \rho$, there are only a finite number of ways to decompose $r$ into Jordan $P$-blocks.
Thus, there are only a finite number of orbits $\mcO_r$ by Theorem~\ref{thm:Hille-Roehrle-block-vec} and Theorem~\ref{thm:Jordan_semicanonical_form}.
Therefore, each fiber of $\pi$ is a union of sets $C_n(\mcO_r)$ for a finite number of orbits $\mcO_r$.
The claim follows from Corollary~\ref{cor:fixed-r-orbit}.
\end{proof}

\begin{theorem}
The variety $C_n$ is connected. Moreover, $C_n$ is irreducible.
\end{theorem}

\begin{proof}
We have $\dim \pi^{-1}(\Delta) \leq 2n - 1$ by Lemma~\ref{lemma:fibers-dimension-pi} and because $\Delta$ is a reducible subvariety of $\CC^n$ of dimension $n-1$.
Let $X = C_n \setminus \pi^{-1}(\Delta) = \pi^{-1}(\CC^n \setminus \Delta)$, 
the complement of $\pi^{-1}(\Delta)$, and $X$ is a connected open subset of $C_n$ by Corollary~\ref{cor:pairwise-distinct-eigenvalues-connected}.
Let $Y$ denote the connected component containing $X$.
If $\pi^{-1}(\Delta)$ is not contained in $Y$, then $C_n$ would have a connected component of dimension less than $2n$, which cannot happen since $\dim(C_n) = 2n$. Hence, we have
\[
Y = \pi^{-1}(\CC^n \setminus \Delta) \cup \pi^{-1}(\Delta) = \pi^{-1}(\CC^n) = C_n,
\]
and so $C_n$ is connected.
\end{proof}

\def\cprime{$'$} \def\cprime{$'$}
\providecommand{\bysame}{\leavevmode\hbox to3em{\hrulefill}\thinspace}
\providecommand{\MR}{\relax\ifhmode\unskip\space\fi MR }
\providecommand{\MRhref}[2]{%
  \href{http://www.ams.org/mathscinet-getitem?mr=#1}{#2}
}
\providecommand{\href}[2]{#2}

\end{document}